\RequirePackage[l2tabu,orthodox]{nag}

\documentclass[11pt,a4paper]{amsart} 

\usepackage[T1]{fontenc}
\usepackage{lmodern}
\usepackage[utf8]{inputenc} 
\usepackage[english]{babel}
\usepackage{amsmath,amsfonts,amssymb,mathrsfs} 
\usepackage{graphicx} 
\usepackage{subcaption}
\usepackage{tikz}
\usepackage[all]{xy}
\usepackage{xcolor}
\usepackage[expansion=false]{microtype} 
\usepackage{hyperref}

\tikzstyle{vertex}=[circle,fill=white,draw,inner sep=0pt,minimum size=5pt]
\tikzstyle{vortex}=[circle,fill=lightgray,draw,inner sep=0pt,minimum size=5pt]
\tikzstyle{vartex}=[circle,fill=black,draw,inner sep=0pt,minimum size=5pt]
\newcommand{\vertex}{\node[vertex]}

\tikzstyle{bvertex}=[circle,fill=white,draw,inner sep=0pt,minimum size=3pt]
\tikzstyle{bvortex}=[circle,fill=lightgray,draw,inner sep=0pt,minimum size=3pt]
\tikzstyle{bvartex}=[circle,fill=black,draw,inner sep=0pt,minimum size=3pt]

\numberwithin{equation}{section}
\numberwithin{figure}{section}

\theoremstyle{plain}
\newtheorem{Thm}{Theorem}[section]

\newtheorem{Prop}[Thm]{Proposition}
\newtheorem{Lemma}[Thm]{Lemma}
\newtheorem{Cor}[Thm]{Corollary}
\newtheorem{ClaimA}{Claim}

\theoremstyle{definition}
\newtheorem{Defn}[Thm]{Definition}

\newtheorem{Conj}[Thm]{Conjecture}
\newtheorem{Ques}[Thm]{Question}
\newtheorem*{Def*}{Definition}

\theoremstyle{remark}
\newtheorem*{Remark}{Remark}
\newtheorem*{Remarks}{Remarks}

\let \seq=\subseteq

\let \0=\emptyset
\let \1=\infty

\let \ga=\gamma
\let \Ga=\Gamma

\let \De=\Delta

\let \p=\ldots
\let \h=\text

\let \cd=\cdots

\newcommand{\WLOGone}{Without loss of generality} 
\newcommand{\WLOGtwo}{without loss of generality} 
\newcommand{\Th}{Theorem}

\newcommand{\Co}{Corollary}

\newcommand{\Pro}{Proposition}
\newcommand{\Pros}{Propositions}

\newcommand{\Con}{Conjecture}
\newcommand{\Que}{Question}

\newcommand{\C}{\mathbf{C}}

\newcommand{\RGa}{R(\Ga_1,\Ga_2)}
\newcommand{\mGa}{\mm(\Ga_1,\Ga_2)}
\newcommand{\RrGa}{R_{\mathrm{red}}(\Ga_1,\Ga_2)}
\newcommand{\RbGa}{R_{\mathrm{blue}}(\Ga_1,\Ga_2)}
\newcommand{\Gr}{G_{\mathrm{red}}}
\newcommand{\Gb}{G_{\mathrm{blue}}}
\newcommand{\Rr}{R_{\mathrm{red}}}
\newcommand{\Rb}{R_{\mathrm{blue}}}
\newcommand{\Cr}{\Cc_{\mathrm{red}}}
\newcommand{\Cb}{\Cc_{\mathrm{blue}}}
\newcommand{\Gg}{\mathcal{G}}
\newcommand{\Hh}{\mathcal{H}}
\renewcommand{\C}{\mathcal{C}}

\newcommand{\mm}{\mathfrak{m}}
\newcommand{\Cc}{\mathfrak{C}}

\title[Generalised Ramsey numbers for two sets of cycles]{Generalised Ramsey numbers \\ for two sets of cycles}
\author{Mikael Hansson}
\thanks{}
\address{Department of Mathematics, Link\"oping University, SE-581 83 Link\"oping, Sweden}
\email{mikael.hansson@liu.se}

\begin{document}

\begin{abstract}
We determine several generalised Ramsey numbers for two sets $\Ga_1$ and $\Ga_2$ of cycles, in particular, all generalised Ramsey numbers $\RGa$ such that $\Ga_1$ or $\Ga_2$ contains a cycle of length at most $6$, or the shortest cycle in each set is even. This generalises previous results of Erd\H{o}s, Faudree, Rosta, Rousseau, and Schelp from the 1970s. Notably, including both $C_3$ and $C_4$ in one of the sets, makes very little difference from including only $C_4$. Furthermore, we give a conjecture for the general case. We also describe many $(\Ga_1,\Ga_2)$-avoiding graphs, including a complete characterisation of most $(\Ga_1,\Ga_2)$-critical graphs, i.e., $(\Ga_1,\Ga_2)$-avoiding graphs on $\RGa-1$ vertices, such that $\Ga_1$ or $\Ga_2$ contains a cycle of length at most $5$. For length $4$, this is an easy extension of a recent result of Wu, Sun, and Radziszowski, in which $|\Ga_1|=|\Ga_2|=1$. For lengths $3$ and $5$, our results are new even in this special case.

\bigskip

\end{abstract}

\maketitle


\section{Introduction} \label{intro}

All graphs in this paper are finite, simple, and undirected. Furthermore, $(G_1,G_2)$ and $(\Gg_1,\Gg_2)$ will always denote a pair of non-empty (uncoloured) graphs and a pair of non-empty sets of non-empty (uncoloured) graphs, respectively. Notation will generally follow \cite{Diestel}.

Here, a \emph{red-blue graph} is a \emph{complete} graph with each edge coloured either red or blue. Red will always be the first colour and blue will always be the second.

Generalised Ramsey numbers for two sets of graphs were, to the best of the author's knowledge, first introduced by Erd\H{o}s, Faudree, Rousseau, and Schelp~\cite{E-F-R-S,F-S2}:

\begin{Defn}
The \emph{generalised Ramsey number} $R(\Gg_1,\Gg_2)$ is the least positive integer $n$, such that each red-blue graph on $n$ vertices contains a red subgraph from $\Gg_1$ or a blue subgraph from $\Gg_2$.
\end{Defn}

Note that when $|\Gg_1|=|\Gg_2|=1$, the generalised Ramsey number $R(\Gg_1,\Gg_2)$ reduces to the ordinary Ramsey number $R(G_1,G_2)$. It is easy to see that $R(\Gg_1,\Gg_2) \leq R(\Hh_1,\Hh_2)$ if each $\Hh_i$ is a non-empty subset of $\Gg_i$. Thus
\begin{equation} \label{ub}
R(\Gg_1,\Gg_2) \leq R(G_1,G_2)
\end{equation}
if each $G_i \in \Gg_i$. In particular, $R(\Gg_1,\Gg_2)$ always exists. Clearly, $R(\Gg_1,\Gg_2)=R(\Gg_2,\Gg_1)$.

\medskip

Let $G$ be a red-blue graph. Then $G$ is called \emph{$(\Gg_1,\Gg_2)$-avoiding} if $G$ contains neither a red subgraph from $\Gg_1$ nor a blue subgraph from $\Gg_2$, and \emph{$(\Gg_1,\Gg_2)$-critical} if, moreover, $G$ has $R(\Gg_1,\Gg_2)-1$ vertices. The \emph{red subgraph} $\Gr$ of $G$ is the (uncoloured) graph $(V(G),\{e \in E(G) \mid \h{$e$ is red}\})$; the \emph{blue subgraph} $\Gb$ is defined analogously. When we say that $G$ is red hamiltonian, blue bipartite, etc., we mean that $\Gr$ is hamiltonian, $\Gb$ is bipartite, and so on. Further terminology will be introduced in Section~\ref{notation}.

Let $(\Ga_1,\Ga_2)$ be a pair of non-empty sets of cycles. The main results of this paper can be divided into two groups: computation of generalised Ramsey numbers $\RGa$ on the one hand, and characterisation of $(\Ga_1,\Ga_2)$-avoiding and $(\Ga_1,\Ga_2)$-critical graphs on the other.

\subsection{Previous results} \label{prev}

We first present some results which, to the best of the author's knowledge, include all previously known generalised Ramsey numbers for two sets of cycles. To this end, and for stating the main results of this paper, let $\C$, $\C_o$, $\C_e$, and $\Cc$ denote the set of all cycles, the set of all odd cycles, the set of all even cycles, and the set of all pairs of non-empty sets of cycles, respectively. Also, for each integer $m \geq 3$, let $\C_{\leq m}=\{C_k \mid k \leq m\}$ and $\C_{\geq m}=\{C_k \mid k \geq m\}$. Furthermore, given $\Ga \seq \C$, let
\[
\min(\Ga)=\begin{cases}\min\{k \mid C_k \in \Ga\} & \h{if $\Ga$ is non-empty,} \\ \infty & \h{otherwise,}\end{cases}
\]
and if $(\Ga_1,\Ga_2) \in \Cc$, then for each $i \in [2]$, let $\ga^i=\min(\Ga_i)$ and $\ga_e^i=\min(\Ga_i \cap \C_e)$.

When $|\Ga_1|=|\Ga_2|=1$, the (ordinary) Ramsey numbers $\RGa$ were determined independently by Rosta~\cite{Rosta} and by Faudree and Schelp~\cite{F-S1}. A new proof, simpler but still quite technical and detailed, was given by K\'arolyi and Rosta~\cite{K-R}. Here we state the theorem in a way that will be useful to us later. To this end, let $\De_1$, $\De_2$, and $\De_3$ be the sets of all pairs $(n,k)$ of integers with $n \geq k \geq 3$, such that
\begin{itemize}
  \item[] $\De_1$: $k \equiv 0$ and either $0 \equiv n \geq 6$ or $n \geq 3k/2$;
  \item[] $\De_2$: $k \equiv 0$ and $1 \equiv n \leq 3k/2$;
  \item[] $\De_3$: $k \equiv 1$ and $n \geq 4$.
\end{itemize}
Here, and henceforth, all congruences are modulo $2$.

\begin{Thm}[\mbox{\cite{Rosta} and \cite{F-S1}}] \label{HS}
Let $n \geq k \geq 3$. Then
\[
R(C_n,C_k)=\begin{cases}6 & \h{if $(n,k) \in \{(3,3),(4,4)\}$,} \\ n+k/2-1 & \h{if $(n,k) \in \De_1$,} \\ 2k-1 & \h{if $(n,k) \in \De_2$,} \\ 2n-1 & \h{if $(n,k) \in \De_3$.}\end{cases}
\]
\end{Thm}

Furthermore, we have the following two results of Erd\H{o}s, Faudree, Rousseau, and Schelp, from which we obtain \Co~\ref{res}.

\begin{Thm}[\mbox{\cite[\Th~3]{E-F-R-S}}] \label{E-F-R-S}
For all $m>n \geq 2$,
\[
R(\C_{\leq m},\{K_n\})=\begin{cases}2n & \h{if $n<m<2n-1$,} \\ 2n-1 & \h{if $m \geq 2n-1$.}\end{cases}
\]
\end{Thm}

\begin{Thm}[\mbox{\cite[\Th~2]{F-S2}}] \label{F-S2}
For all $m \geq 3$ and all $n \geq 2$,
\[
R(\C_{\geq m},\{K_n\})=(m-1)(n-1)+1.
\]
\end{Thm}

\begin{Cor} \label{res}
For all $m \geq 3$,
\[
R(\C_{\leq m},\{C_3\})=\begin{cases}6 & \h{if $m \leq 4$,} \\ 5 & \h{if $m \geq 5$,}\end{cases}
\]
and
\[
R(\C_{\geq m},\{C_3\})=2m-1.
\]
\end{Cor}

\medskip

Let us turn to a structural result, due to Wu, Sun, and Radziszowski~\cite{W-S-R}. In order to state it, we have to define some sets of graphs. Given a graph $G$, a vertex $v \in V(G)$, a subset $U \seq V(G)$, and an edge $e \in E(G)$, let $N(v)$, $d(v)$, $G[U]$, and $G-e$ denote the set of neighbours of $v$, the number of neighbours of $v$, the induced subgraph on $U$, and the graph $G$ with the edge $e$ deleted, respectively. We say that $G$ has a \emph{matching on $U$} if each vertex in $U$ has degree at most $1$ in $G[U]$.

\begin{Defn}
Let $n \geq 6$. For each $i \in [3]$, let $\Gg_i$ be a set of graphs on $\{v\} \cup X \cup \{y\}$ with $X=\{x_1,\p,x_{n-2}\}$, as follows:
\begin{itemize}
  \item $\Gg_1$ consists of the graphs with $N(v)=X$, having a matching on $X \cup \{y\}$;
  \item $\Gg_2$ consists of the graphs with $N(v)=X$, $N(y)=\{x_{n-2}\}$, and $d(x_{n-2})=3$, having a matching on $X$;
  \item $\Gg_3$ consists of the graphs with $N(v)=X \cup \{y\}$, having a matching on $X \cup \{y\}$.
\end{itemize}
See \cite[Figure~1]{W-S-R} for pictures illustrating these graph sets.
\end{Defn}

We can now state the structural result:

\begin{Thm}[\mbox{\cite[\Th~1]{W-S-R}}] \label{W-S-R}
Let $n \geq 6$. Then $G$ is $(C_n,C_4)$-critical if and only if $\Gb \in \Gg_1 \cup \Gg_2 \cup \Gg_3$.
\end{Thm}

\subsection{Main results} \label{main}

Recall that $\ga^i$ and $\ga_e^i$ denote the length of the shortest cycle and the length of the shortest even cycle, respectively, in $\Ga_i$. We shall later (in Subsection~\ref{redbluenumbers}) define a number $\mm=\mGa$, and we shall prove that
\[
\mm=\max\!\big(5,\min(\ga^2+\ga_e^1/2-1,2\ga^2-1),\min(\ga^1+\ga_e^2/2-1,2\ga^1-1)\big).
\]
Recall also that $\Cc$ is the set of all pairs of non-empty sets of cycles, and let $\mathscr{C}=\{(\Ga_1,\Ga_2) \in \Cc \mid \h{$C_3$ or $C_4 \in \Ga_1 \cap \Ga_2$ and $C_3$ or $C_5 \notin \Ga_1 \cup \Ga_2$}\}$. Observe that $\mGa=5$ if $(\Ga_1,\Ga_2) \in \mathscr{C}$.

We are now ready to state the first of two main results of this paper:

\begin{Thm} \label{MT}
Let $(\Ga_1,\Ga_2) \in \Cc$. Then
\begin{equation} \label{eq:MT}
\RGa \geq \begin{cases}\mGa+1 & \h{if $(\Ga_1,\Ga_2) \in \mathscr{C}$,} \\ \mGa & \h{otherwise.}\end{cases}
\end{equation}
Moreover, equality holds if either $\min(\Ga_1 \cup \Ga_2) \leq 6$ or $(\Ga_1,\Ga_2) \in \Cc_1 \cup \Cc_2$.
\end{Thm}

We shall prove \Th~\ref{MT} in Subsection~\ref{numbers}. The definitions of $\Cc_1$ and $\Cc_2$ are slightly involved; they will be given before \Pro~\ref{C1C2}, in which they are needed. However, an easily stated special case is the following:

\begin{Cor} \label{cor1}
Equality holds in \eqref{eq:MT} for all pairs $(\Ga_1,\Ga_2) \in \Cc$ such that both $\ga^1$ and $\ga^2$ are even.
\end{Cor}

Choose $\Ga_1 \seq \C$ and $\Ga_2 \seq \C_{\geq 5}$ arbitrarily, and let $R_1=R(\Ga_1,\{C_4\} \cup \Ga_2)$ and $R_2=R(\Ga_1,\C_{\leq 4} \cup \Ga_2)$. Since equality holds in \eqref{eq:MT} if $\min(\Ga_1 \cup \Ga_2) \leq 6$, we have the following:

\begin{Cor} \label{cor2}
Let $\Ga_1$, $\Ga_2$, $R_1$, and $R_2$ be as above. Then $R_1=R_2$ whenever $\ga^1 \geq 6$; in particular, $R(C_n,C_4)=R(\{C_n\},\C_{\leq 4})$ whenever $n \geq 6$.
\end{Cor}

Hence, including both $C_3$ and $C_4$ in one of the sets, rarely makes any difference from including only $C_4$. This is perhaps quite unintuitive. Using \Th~\ref{MT}, it is possible to determine precisely under what conditions $R_1=R_2$; we refrain from doing this here.

For all pairs $(\Ga_1,\Ga_2)$ such that equality holds in \eqref{eq:MT}, note that since $\mGa$ only depends on $\ga^1$, $\ga_e^1$, $\ga^2$, and $\ga_e^2$, so does $\RGa$, unless $(\ga^1,\ga^2) \in \{(3,3),(4,3),(3,4)\}$.

Before we turn to the structural results, let us state a conjecture:

\begin{Conj} \label{MC}
Equality holds in \eqref{eq:MT} for all pairs $(\Ga_1,\Ga_2) \in \Cc$.
\end{Conj}

\medskip

Let $(\Ga_1,\Ga_2) \in \Cc$ with $\ga^1 \geq 6$ and $\ga^2=4$, and let $n=\ga^1$. Then, since $\RGa=n+1=R(C_n,C_4)$, and none of the graphs in $\Gg_1 \cup \Gg_2 \cup \Gg_3$ contains a cycle of length at least $4$ or has a cycle of length at least $n$ in its complement, the following result is a direct consequence of \Th~\ref{W-S-R}.

\begin{Thm} \label{struc4}
Let $(\Ga_1,\Ga_2) \in \Cc$ with $\ga^1 \geq 6$ and $\ga^2=4$, and let $n=\ga^1$. Then $G$ is $(\Ga_1,\Ga_2)$-critical if and only if $\Gb \in \Gg_1 \cup \Gg_2 \cup \Gg_3$.
\end{Thm}

We shall prove a similar result when $\ga^2 \in \{3,5\}$, which is the second main result of this paper. Here we state an abridged version.

\begin{Thm}[abridged version of \Th~\ref{struc35+}] \label{struc35}
Given $(\Ga_1,\Ga_2) \in \Cc$ and a $(\Ga_1,\Ga_2)$-critical graph $G$, assume that either
\begin{equation} \label{eq1}
\h{$\ga^1 \geq 5$ and $\ga^2=3$}
\end{equation}
or
\begin{equation} \label{eq2}
\h{$\ga^1 \geq 6$, $\ga^2=5$, and $\ga_e^2 \geq 8$.}
\end{equation}
Then $G$ is blue bipartite.
\end{Thm}

\Th~\ref{struc35+} will give us a complete characterisation of the $(\Ga_1,\Ga_2)$-critical graphs that satisfy \eqref{eq1} or \eqref{eq2}. It is stated and proved in Subsection~\ref{struc}. From \Th~\ref{struc35+}, we easily obtain the following result, which characterises $(C_n,C_3)$- and $(C_n,C_5)$-critical graphs, and is the analogue of \Th~\ref{W-S-R}. Note that there are a lot fewer possibilities for the critical graphs, compared to the situation in \Th~\ref{W-S-R}.

\begin{Cor} \label{cor3}
${}$
\begin{itemize}
  \item[(a)] Let $n \geq 5$. Then $G$ is $(C_n,C_3)$-critical if and only if $\Gb=K_{n-1,n-1}$ or $K_{n-1,n-1}-e$ for some edge $e$.
  \item[(b)] Let $n \geq 6$. Then $G$ is $(C_n,C_5)$-critical if and only if $\Gb=K_{n-1,n-1}$ or $K_{n-1,n-1}-e$ for some edge $e$.
\end{itemize}
\end{Cor}

\begin{Remark}
In particular, for $n \geq 6$, a graph is $(C_n,C_3)$-critical precisely when it is $(C_n,C_5)$-critical.
\end{Remark}

We pose the following question:

\begin{Ques} \label{ques}
Let $n \geq k \geq 3$, where $k$ is odd and $(n,k) \neq (3,3)$. For which pairs $(n,k)$ is it true that $G$ is $(C_n,C_k)$-critical if and only if $\Gb=K_{n-1,n-1}$ or $K_{n-1,n-1}-e$ for some edge $e$? Is it true for all of them?
\end{Ques}

\medskip

After a short section on notation and conventions (Section~\ref{notation}), the paper is organised as follows. In Section~\ref{prel}, we define some colourings that we shall need later, and some numbers from which $\mGa$ stems. We also prove some preparatory results. Then, in Section~\ref{proofs}, we prove the two main results of this paper. Subsection~\ref{numbers} is devoted to computation of generalised Ramsey numbers (thus establishing \Th~\ref{MT}), while Subsection~\ref{struc} deals with characterisation of $(\Ga_1,\Ga_2)$-avoiding and $(\Ga_1,\Ga_2)$-critical graphs (in particular, \Th~\ref{struc35+}).

\section{Notation and conventions} \label{notation}

A graph on $n \geq 3$ vertices is called \emph{hamiltonian} if it contains a cycle of length $n$ (an \emph{$n$-cycle}), and \emph{pancyclic} if it contains cycles of every length between $3$ and $n$. Vertex indices will always be interpreted modulo the length of the cycle that we are considering at the moment. For instance, $x_{11}=x_3$ in a cycle of length $8$. When we consider two vertices $x_i$ and $x_{i+j}$ of a cycle $C=x_1x_2 \cd x_nx_1$, we shall assume that $0 \leq j \leq n-1$. A \emph{$j$-chord} of $C$ is an edge of the form $x_ix_{i+j}$, where $2 \leq j \leq n-2$.

Let $G$ be a red-blue graph. Two vertices $u,v \in V(G)$ are called \emph{red adjacent} to each other, or \emph{red neighbours}, if the edge $uv$ is red; \emph{blue adjacent} is defined analogously. If $U \seq V(G)$, let $G[U]$ denote the induced subgraph on $U$ with the induced colouring. Also, if $H=G[U]$ and $U' \seq V(G)$, let $H-U'=G[U-U']$, and if $v \in V(G)$, let $H-v=H-\{v\}$. When there is no risk of ambiguity, we write $v \in G$ and $|G|$ instead of $v \in V(G)$ and $|V(G)|$, respectively.

\section{Preliminaries} \label{prel}

In this section, we first define some colourings that will be needed in the proofs to come. We then define bipartite versions of Ramsey numbers, and compute them for any sets of cycles. They turn out to be closely related to generalised Ramsey numbers (see \Th~\ref{MT}). Finally, we state and prove a number of lemmas. They will be used in the proofs of \Pros~\ref{C1C2}, \ref{g1>=5g2=3}, and \ref{g1>=6g2=5}, which establish \Th~\ref{MT} when $(\Ga_1,\Ga_2) \in \Cc_1 \cup \Cc_2$, $\ga^1 \geq 5$ and $\ga^2=3$, and $\ga^1 \geq 6$ and $\ga^2=5$, respectively.

\bigskip

We shall need the following colourings:

\medskip

\textbf{Colouring~1:} The red-blue graph on $4$ vertices whose red (and blue) \indent subgraph is a path of length $3$.

\medskip

\textbf{Colouring~2:} The red-blue graph on $5$ vertices whose red (and blue) \indent subgraph has vertex degrees $1$, $1$, $2$, $3$, and $3$.

\medskip

\textbf{Colouring~3:} The red-blue graph on $5$ vertices whose red (and blue) \indent subgraph is a cycle of length $5$.

\medskip

\textbf{Colouring~4:} The red-blue graph on $n+k/2-2$ vertices, with $n \geq k/2$ \indent and $k$ even, whose blue subgraph equals $K_{n-1,k/2-1}$.

\medskip

\textbf{Colouring~5:} The red-blue graph on $2k-2$ vertices whose red subgraph \indent equals $K_{k-1,k-1}$.

\medskip

\textbf{Colouring~6:} The red-blue graph on $2n-2$ vertices whose blue subgraph \indent equals $K_{n-1,n-1}$.

\bigskip

Note that Colourings~3, 4, 5, and 6 were used to prove the lower bounds in \Th~\ref{HS} (see \cite{F-S1} or \cite{Rosta}).

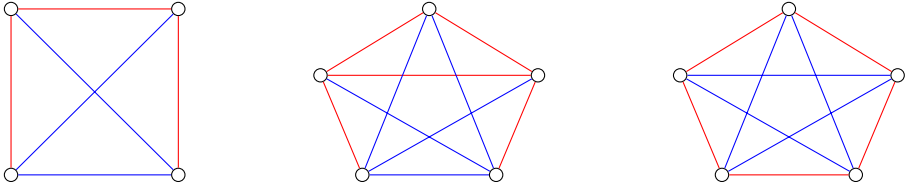
\begin{figure}[tb]
\begin{tikzpicture}[scale=2.2]
  \vertex (00) at (0,0) {}; \vertex (01) at (0,1) {}; \vertex (10) at (1,0) {}; \vertex (11) at (1,1) {};
  \draw[red]  (00)--(01)--(11)--(10);
  \draw[blue] (01)--(10)--(00)--(11);
  \vertex (1c) at (2.1,0) {}; \vertex (1d) at (2.9,0) {}; \vertex (1b) at (1.85,0.6) {}; \vertex (1e) at (3.15,0.6) {}; \vertex (1a) at (2.5,1) {};
  \vertex (2c) at (4.25,0) {}; \vertex (2d) at (5.05,0) {}; \vertex (2b) at (4,0.6) {}; \vertex (2e) at (5.3,0.6) {}; \vertex (2a) at (4.65,1) {};
  \draw[red]  (1c)--(1b)--(1a)--(1e)--(1d); \draw[red]  (1b)--(1e);
  \draw[blue] (1b)--(1d)--(1a)--(1c)--(1e); \draw[blue] (1c)--(1d);
  \draw[red]  (2a)--(2b)--(2c)--(2d)--(2e)--(2a);
  \draw[blue] (2a)--(2d)--(2b)--(2e)--(2c)--(2a);
\end{tikzpicture}
\caption{Colourings~1, 2, and 3.} \label{fig}
\end{figure}

\subsection{Bipartite versions of Ramsey numbers} \label{redbluenumbers}

\begin{Defn} \label{redbluedefn}
Let $\Rr(\Gg_1,\Gg_2)$ be the least positive integer $n$, such that each red bipartite graph on $n$ vertices contains a red subgraph from $\Gg_1$ or a blue subgraph from $\Gg_2$; $\Rb(\Gg_1,\Gg_2)$ is defined analogously.
\end{Defn}

Observe that
\[
\Rr(\Gg_1,\Gg_2)=R(\Gg_1 \cup \C_o,\Gg_2) \leq R(\Gg_1,\Gg_2)
\]
and
\[
\Rb(\Gg_1,\Gg_2)=R(\Gg_1,\Gg_2 \cup \C_o) \leq R(\Gg_1,\Gg_2).
\]
In particular, $\Rr(\Gg_1,\Gg_2)$ and $\Rb(\Gg_1,\Gg_2)$ always exist. Moreover, the following fact, which we shall use frequently, follows:
\begin{equation} \label{R-comp}
R(\Gg_1,\Gg_2) \geq \max\!\big(\Rr(\Gg_1,\Gg_2),\Rb(\Gg_1,\Gg_2)\big).
\end{equation}

\begin{Prop} \label{redblue}
Let $(\Ga_1,\Ga_2) \in \Cc$. Then
\begin{equation} \label{blue}
\RbGa=\begin{cases}\ga^1+\ga_e^2/2-1 & \h{if $2\ga^1>\ga_e^2$ and $(\ga^1,\ga_e^2) \neq (3,4)$,} \\ 2\ga^1-1 & \h{if $2\ga^1 \leq \ga_e^2$ or $(\ga^1,\ga_e^2)=(3,4)$,}\end{cases}
\end{equation}
or equivalently,
\[
\RbGa=\begin{cases}5 & \h{if $(\ga^1,\ga_e^2)=(3,4)$,} \\ \min(\ga^1+\ga_e^2/2-1,2\ga^1-1) & \h{otherwise.}\end{cases}
\]
\end{Prop}

\begin{Remarks}
Of course, the analogous result holds for $\RrGa$. Furthermore, one can extend the definition of $\RbGa$ to include the case $\Ga_2=\emptyset$, and \eqref{blue} will still hold (note that $\ga^2=\ga_e^2=\infty$). Also note that $\ga^1+\ga_e^2/2-1=2\ga^1-1$ when $2\ga^1=\ga_e^2$.
\end{Remarks}

\begin{proof}
In order to simplify notation, let $n=\ga^1$ and $k=\ga_e^2$. Let us first prove the lower bounds.

$2n>k$: Colouring~4 is blue bipartite on $n+k/2-2$ vertices, and contains no red cycle of length at least $n$, no blue cycle of length at least $k$, and no odd blue cycle. Hence, $\RbGa \geq n+k/2-1$.

$2n \leq k$ or $(n,k)=(3,4)$: Colouring~6 is blue bipartite on $2n-2$ vertices, and contains no red cycle of length at least $n$, no blue cycle of length at least $k$, and no odd blue cycle. Hence, $\RbGa \geq 2n-1$ when $2n \leq k$. When $(n,k)=(3,4)$, the lower bound follows from Colouring~1.

We now turn to the upper bounds. Let $G$ be an arbitrary blue bipartite graph on $2n-1$ vertices; say that $\Gb \seq K_{p,q}$, where $p+q=2n-1$. Then $\max(p,q) \geq n$, whence $G$ contains a red $C_n$. Hence, we only have to consider the case when $2n>k$ and $(n,k) \neq (3,4)$.

Since $2n>k$ and $(n,k) \neq (3,4)$, $n \geq 4$. Let $G$ be an arbitrary blue bipartite graph on $n+k/2-1$ vertices; say that $\Gb \seq K_{p,q}$, where $p+q=n+k/2-1$ and $p \geq q$, and assume $G$ is $(\Ga_1,\Ga_2)$-avoiding. Were $p \geq n$, $G$ would contain a red $C_n$, whence $q \geq k/2$. Since $G$ contains no blue $C_k$, there is at least one red edge between the red $K_p$ and the red $K_q$. Were there two disjoint red edges between $K_p$ and $K_q$, $G$ would contain a red $C_n$ (since $n \geq 4$), whence all red edges between $K_p$ and $K_q$ have a common vertex $x$. Since $G$ contains no blue $C_k$, it follows that $q=k/2$ and $p=n-1$.

Assume first that $x \in K_{k/2}$. Were there at least two red edges between $K_{n-1}$ and $x$, $G$ would contain a red $C_n$, whence there is only one red edge between $K_{n-1}$ and $x$. Hence, there are at least two blue edges between $K_{n-1}$ and $x$ (since $n \geq 4$), say $v_1x$ and $v_2x$, and $v_1xv_2$ can then be extended to a blue $C_k$, contrary to the hypothesis.

Assume now that $x \in K_{n-1}$. Then all edges between $K_{n-1}-x$ and $K_{k/2}$ are blue, whence $G$ contains a blue $C_k$, unless $n-1=k/2$, in which case $x \in K_{k/2}$ (which we have already treated).
\end{proof}

Given $(\Ga_1,\Ga_2) \in \Cc$, let
\[
\mm=\mGa=\max\!\big(\RrGa,\RbGa\big).
\]
The following result is an immediate consequence of \Pro~\ref{redblue}.

\begin{Cor} \label{maxredblue}
Let $(\Ga_1,\Ga_2) \in \Cc$. Then
\[
\mm=\max\!\big(5,\min(\ga^2+\ga_e^1/2-1,2\ga^2-1),\min(\ga^1+\ga_e^2/2-1,2\ga^1-1)\big).
\]
\end{Cor}

\subsection{Preparatory results} \label{prep}

The following three lemmas, due to K\'arolyi and Rosta~\cite{K-R}, guarantee the existence of certain monochromatic cycles under various assumptions. They will be used in the proof of \Pro~\ref{C1C2}.

\begin{Lemma}[\mbox{\cite[Lemma~3.1]{K-R}}] \label{Lemma 3.1}
Let $G$ be a red-blue graph on $n+k/2-1$ vertices, where $n \geq k \geq 4$, $n \geq 5$, and $k \equiv 0$. Then $G$ contains either a monochromatic cycle of length at least $n$ or a blue $C_k$.
\end{Lemma}

\begin{Remark}
Of course, if $G$ contains no monochromatic cycle of length at least $n$, then we also have a red $C_k$. However, we only need the lemma as stated above.
\end{Remark}

\begin{Lemma}[\mbox{\cite[Lemma~2.1]{K-R}}] \label{Lemma 2.1}
Let $G$ be a red-blue graph.
\begin{itemize}
  \item[(a)] If $G$ contains a monochromatic $C_{2n+1}$ for some $n \geq 3$, then $G$ contains a monochromatic $C_{2n}$.
  \item[(b)] If $G$ contains a monochromatic $C_{2n}$ for some $n \geq 3$, then $G$ contains a monochromatic $C_{2n-2}$.
\end{itemize}
\end{Lemma}

\begin{Lemma}[special case of \mbox{\cite[Lemma~3.3]{K-R}}] \label{Lemma 3.3}
Let $n \geq k \geq 4$ with $k \equiv 0$. If a red-blue graph contains a blue $C_n$, then it contains either a red $C_n$ or a blue $C_k$.
\end{Lemma}

The following two results will be used in the proof of \Pro~\ref{g1>=5g2=3}.

\begin{Lemma} \label{Lemma X}
Let $G$ be a red-blue graph on $n$ vertices without blue $3$-cycles. Then $G$ is either red hamiltonian or blue bipartite.
\end{Lemma}

\begin{proof}
We use induction on $n$. If $n \leq 4$, then $G$ is blue bipartite. Assume now that the statement holds for some $n \geq 4$, and let $|G|=n+1$. If $G$ is not blue bipartite, let $C=x_1x_2 \cd x_{2m+1}x_1$ be a shortest odd blue cycle in $G$; note that $m \geq 2$.

Were some chord of $C$ blue, $G$ would contain an odd blue cycle shorter than $C$, whence all chords of $C$ are red. In particular, the $m$-chords of $C$ form a red $C_{2m+1}$. If $V(C)=V(G)$, this is a red $C_{n+1}$. If not, take $x \in G-V(C)$. By the induction hypothesis, $G-x$ has a red $C_n$, say $v_1v_2 \cd v_nv_1$.

Assume $G$ is not red hamiltonian. Suppose $xv_i$ and $xv_j$ are red; note that $|i-j| \geq 2$. If $v_{i+1}v_{j+1}$ is red, then $G$ has a red $C_{n+1}$. Hence, $G$ contains a blue $d$-clique, where $d$ is the number of red neighbours of $x$. Therefore, since $G$ contains no blue $C_3$, $d \leq 2$. Hence, $x$ is blue adjacent to two consecutive vertices of $C$, yielding a blue $C_3$, a contradiction.
\end{proof}

\begin{Lemma} \label{chords}
Let $G$ be a red-blue graph and let $n \geq 6$. If $G$ contains an $n$-cycle $C$, all of whose chords are red, then $G[V(C)]$ is red pancyclic.
\end{Lemma}

In order to get a short proof, we use the following result of Bondy.

\begin{Lemma}[\mbox{\cite[\Th~1]{Bondy}}] \label{Bondy}
Let $G$ be hamiltonian with $n$ vertices and at least $n^2/4$ edges. Then either $G$ is pancyclic or $G=K_{n/2,n/2}$.
\end{Lemma}

\begin{proof}[Proof of Lemma~\ref{chords}]
Let $C=x_1x_2 \cd x_nx_1$. Then
\[
x_1x_3 \cd x_nx_2x_4 \cd x_{n-1}x_1
\]
is a red $C_n$ if $n$ is odd, and
\[
x_1x_3 \cd x_{n-1}x_2x_nx_{n-2} \cd x_4x_1
\]
is a red $C_n$ if $n$ is even. Since $x_1x_3x_5x_1$ is a red $C_3$, $G$ is not red bipartite. Hence, it follows from Lemma~\ref{Bondy} that $G[V(C)]$ is red pancyclic.
\end{proof}

Our next result essentially proves \Pro~\ref{g1>=6g2=5}.

\begin{Prop} \label{new}
Suppose $\Gb$ contains an odd cycle but no $5$-cycle. Then $\Gr$ contains cycles of every length in the interval $[6,|G|-2]$.
\end{Prop}

In order to prove it, we need the following three lemmas.

\begin{Lemma} \label{new1}
If $\Gb$ contains no $5$-cycle and $\Gr$ has an $n$-cycle, where $n \geq 5$, then $\Gr$ has an $(n-1)$-cycle or an $(n-2)$-cycle.
\end{Lemma}

\begin{Remark}
The assumption on $n$ is clearly necessary.
\end{Remark}

\begin{proof}
Let $C$ be a red $C_n$. If $C$ has a red $2$-chord or $3$-chord, then we have a red $C_{n-1}$ or $C_{n-2}$, respectively. Otherwise we have a blue $C_5$.
\end{proof}

\begin{Lemma} \label{new2}
If $\Gb$ contains no $5$-cycle and $\Gr$ has an $n$-cycle, where $7 \leq n \leq |G|-1$, then $\Gr$ has an $(n-1)$-cycle.
\end{Lemma}

\begin{Remark}
The red-blue graph whose red subgraph equals $K_{4,4}$ shows that the conditions $n \geq 7$ and $n \leq |G|-1$ cannot be omitted.
\end{Remark}

\begin{proof}
Let $C=x_1x_2 \cd x_nx_1$ be a red $C_n$, and take $x \in G-V(C)$. If $C$ has a red $2$-chord, then we have a red $C_{n-1}$, so suppose all $2$-chords are blue. Then, if $G$ contains no blue $C_5$, $x$ is red adjacent to some vertex of $C$. Hence, either there is an $i \in [n]$ such that $xx_i$ and $xx_{i+3}$ are red, or there is an $i \in [n]$ such that $xx_i$ and $xx_{i+6}$ are blue. In the former case, we have a red $C_{n-1}$, and in the latter case, we have a blue $C_5$.
\end{proof}

\begin{Lemma} \label{old}
Suppose $\Gb$ contains no $5$-cycle, $C=x_1x_2 \cd x_nx_1$ is a red $n$-cycle, and $x \in G-V(C)$. If there is a $j \in \{1\} \cup [3,n-3] \cup \{n-1\}$, such that for some $i \in [n]$, $xx_i$ and $xx_{i+j}$ are red, then $G$ contains a red $(n+1)$-cycle.
\end{Lemma}

\begin{proof}
This is clear if $j \in \{1,n-1\}$, so take $j \in [3,n-3]$. In order to obtain a contradiction, suppose $G$ does not contain a red $C_{n+1}$. Since $j \in [3,n-3]$, $x_{i-1}$, $x_i$, $x_{i+1}$, $x_{i+j-1}$, $x_{i+j}$, and $x_{i+j+1}$ are all distinct, and since $G$ contains no red $C_{n+1}$, $x$ is blue adjacent to $x_{i-1}$, $x_{i+1}$, $x_{i+j-1}$, and $x_{i+j+1}$. Were $x_{i-1}x_{i+j-1}$ or $x_{i+1}x_{i+j+1}$ red,
\[
x_{i-1}x_{i+j-1}x_{i+j-2} \cd x_ixx_{i+j}x_{i+j+1} \cd x_{i-1}
\]
or
\[
x_{i+1}x_{i+j+1}x_{i+j+2} \cd x_ixx_{i+j}x_{i+j-1} \cd x_{i+1},
\]
respectively, would be a red $C_{n+1}$, whence $x_{i-1}x_{i+j-1}$ and $x_{i+1}x_{i+j+1}$ are blue. Since $G$ contains no blue $C_5$, $x_{i-1}x_{i+1}$ is red. Were $x_ix_{i+j-1}$ and $x_ix_{i+j+1}$ blue, $x_{i+j-1}x_ix_{i+j+1}x_{i+1}xx_{i+j-1}$ would be a blue $C_5$, whence $x_ix_{i+j-1}$ or $x_ix_{i+j+1}$ is red. Thus
\[
x_ix_{i+j-1}x_{i+j-2} \cd x_{i+1}x_{i-1}x_{i-2} \cd x_{i+j}xx_i
\]
or
\[
x_ix_{i+j+1}x_{i+j+2} \cd x_{i-1}x_{i+1}x_{i+2} \cd x_{i+j}xx_i,
\]
respectively, is a red $C_{n+1}$, contrary to the hypothesis.
\end{proof}

We can now prove \Pro~\ref{new}.

\begin{proof}[Proof of \Pro~\ref{new}]
We may assume $|G| \geq 8$. Since $R(C_4,C_5)=7$, $G$ contains a red $C_4$. We shall prove the following:

\begin{ClaimA} \label{max}
If $C$ is a red cycle of maximal length, then $|G-V(C)| \leq 2$.
\end{ClaimA}

Now, if $|G-V(C)|=0$, then, by Lemma~\ref{new1}, $G$ contains a red $C_{|G|-1}$ or a red $C_{|G|-2}$. In the former case, we also have a red $C_{|G|-2}$, by Lemma~\ref{new2}. If $|G-V(C)|=1$, then, by Lemma~\ref{new2}, we have a red $C_{|G|-2}$, which is also the case if $|G-V(C)|=2$. Hence, by Claim~\ref{max}, we always have a red $C_{|G|-2}$. Therefore, by Lemma~\ref{new2}, \Pro~\ref{new} follows.

\medskip

It remains to prove Claim~\ref{max}. Thus, let $C=x_1x_2 \cd x_nx_1$ be a red cycle of maximal length, and let $v_1,\p,v_k$ be the vertices of $G-V(C)$; assume $k \geq 3$. We shall prove that $G$ contains a red cycle longer than $C$, whence, in fact, $k \leq 2$.

Assume first that some $v_i$, say $v_1$, is red adjacent to at least two vertices of $C$. By Lemma~\ref{old}, we may assume that $v_1x_2$ and $v_1x_4$ are red. Then $v_1x_1$, $v_1x_3$, $v_1x_5$, $x_1x_3$, and $x_3x_5$ are blue. We consider three cases:

$n \geq 7$ or $n=5$: Were $v_2$ red adjacent either to $x_1$ and $x_3$ or to $x_3$ and $x_5$, $G$ would contain a red $C_{n+2}$. It follows from Lemma~\ref{old} that were $v_2$ red adjacent to $x_1$ and $x_5$, $G$ would contain a red $C_{n+1}$. Hence, $v_2$ has at most one red edge to $\{x_1,x_3,x_5\}$, whence we obtain a blue $C_5$, a contradiction.

$n=6$: Since $G$ contains no blue $C_5$, $v_2$ cannot be blue adjacent to both $x_1$ and $x_5$; say that $v_2x_5$ is red. Then $v_2x_6$ is blue. Were $v_2x_3$ red, we would obtain a red $C_{n+2}$, whence $v_2x_3$ is blue. Since $G$ contains no blue $C_5$, $v_2x_1$ is red. Now, if $v_1x_6$ is blue, then $v_1x_6v_2x_3x_1v_1$ is a blue $C_5$, and if $v_1x_6$ is red, then $v_1x_6x_1v_2x_5x_4x_3x_2v_1$ is a red $C_{n+2}$, a contradiction.

$n=4$: It is easily seen that $v_2$, as well as $v_3$, has at most one red edge to $\{v_1,x_1,x_3\}$. Were $v_2$ and $v_3$ not red adjacent to the same of these three vertices, $G$ would contain a blue $C_5$, whence this is not the case. Now, if $v_2x_2$ and $v_3x_2$ are blue, we obtain a blue $C_5$, and if either $v_2x_2$ or $v_3x_2$ is red, we have a red $C_{n+1}$, a contradiction.

Assume now that each $v_i$ is red adjacent to at most one vertex of $C$. Then every three vertices of $G-V(C)$ belong to a common blue $C_6$ whose vertices alternate between $G-V(C)$ and $C$. Therefore, a blue edge in $G-V(C)$ would yield a blue $C_5$, whence all $v_iv_j$ are red. Hence, since $G$ is not blue bipartite, $G[V(C)]$ contains a blue edge $x_ix_j$. If both $x_i$ and $x_j$ have blue neighbours in $G-V(C)$, then there are $v,v' \in G-V(C)$ and a blue path $vx_ix_jv'$ which extends to a blue $C_5$. Hence, $x_i$ (say) has only red neighbours in $G-V(C)$. This accounts for all red edges between $G-V(C)$ and $C$. Therefore, every blue edge in $G[V(C)]$ is incident with $x_i$. Hence, $G$ is blue bipartite, with parts $(G-V(C)) \cup \{x_i\}$ and $G[V(C)-\{x_i\}]$, a contradiction. This completes the proof.
\end{proof}

\section{Proofs of the main results} \label{proofs}

\subsection{Generalised Ramsey numbers} \label{numbers}

In this subsection, we prove \Th~\ref{MT}. The proof consists of a number of propositions, the first of which (\Pro~\ref{C1C2}) takes care of the case $(\Ga_1,\Ga_2) \in \Cc_1 \cup \Cc_2$; the others (\Pros~\ref{C3C3orC4C4}--\ref{g1>=6g2=6}) deal with the case $\min(\Ga_1 \cup \Ga_2) \leq 6$.

In order to state \Pro~\ref{C1C2}, we have to define $\Cc_1$ and $\Cc_2$. Let $\Cc_1$ be the set of all pairs $(\Ga_1,\Ga_2) \in \Cc$ such that at least one of the following conditions holds:
\begin{itemize}
  \item[(i)]   $0 \equiv \ga^2 \geq \max(6,\ga_e^1)$ or $\ga^2 \geq 3\ga_e^1/2$;
  \item[(ii)]  $\ga^1 \equiv 1$, $\ga_e^1 \geq 2\ga^2$, and either $0 \equiv \ga^2 \geq 2\ga^1/3$ or $\ga^2 \geq \max(4,\ga^1)$;
  \item[(iii)] $\ga^2>\ga_e^1$ and $\ga_e^2=\ga^2+1$.
\end{itemize}
Also, let $\Cc_2$ be as $\Cc_1$, but with the roles of $1$ and $2$ interchanged in the superscripts.

\begin{Prop} \label{C1C2}
Let $(\Ga_1,\Ga_2) \in \Cc_1 \cup \Cc_2$. Then
\[
\RGa=\mGa=\begin{cases}\RrGa & \h{if $(\Ga_1,\Ga_2) \in \Cc_1$,} \\ \RbGa & \h{if $(\Ga_1,\Ga_2) \in \Cc_2$.}\end{cases}
\]
\end{Prop}

\begin{proof}
For each $j \in [4]$, let $\Cr^j$ consist of all pairs $(\Ga_1,\Ga_2) \in \Cc$ such that
\begin{itemize}
  \item[] $\Cr^1$: $0 \equiv \ga^2 \geq \max(6,\ga_e^1)$ or $\ga^2 \geq 3\ga_e^1/2$;
  \item[] $\Cr^2$: $0 \equiv \ga^2<\ga^1$, $1 \equiv \ga^1 \leq 3\ga^2/2$, and $\ga_e^1 \geq 2\ga^2$;
  \item[] $\Cr^3$: $1 \equiv \ga^1 \leq \ga^2 \leq \ga_e^1/2$ and $\ga^2 \geq 4$;
  \item[] $\Cr^4$: $\ga^2>\ga_e^1$ and $\ga_e^2=\ga^2+1$.
\end{itemize}
Furthermore, let $\Cb^j$ be as $\Cr^j$, but with the roles of $1$ and $2$ interchanged in the superscripts. It is easily seen that
\[
\Cc_1=\bigcup_{j=1}^4{\Cr^j} \quad \h{and} \quad \Cc_2=\bigcup_{j=1}^4{\Cb^j}.
\]

\medskip

Choose arbitrary sets $\Phi \seq \C_{\geq n}$, $\Psi \seq \C_{\geq k}$, $\Phi_0 \seq \C_{\geq 2k}$, $\Psi_0 \seq \C_{\geq 2n}$, and $\Omega \seq \C_o$, where $n \geq k \geq 3$. Observe that Colouring~4 contains no red member of $\Phi$ and no blue member of $\Psi \cup \Omega$. By \eqref{ub} and the upper bounds in \Th~\ref{HS}, we therefore conclude that
\begin{equation} \label{eq:1}
R(\{C_n\} \cup \Phi,\{C_k\} \cup \Psi \cup \Omega)=n+k/2-1 \quad \h{if $(n,k) \in \De_1$.}
\end{equation}
By considering Colourings~5 and 6 in the same way, we also obtain
\begin{align}
R(\{C_n\} \cup \Phi_0 \cup \Omega,\{C_k\} \cup \Psi)=2k-1 &\quad \h{if $(n,k) \in \De_2$} \label{eq:2} \\
\intertext{and}
R(\{C_n\} \cup \Phi,\{C_k\} \cup \Psi_0 \cup \Omega)=2n-1 &\quad \h{if $(n,k) \in \De_3$.} \label{eq:3}
\end{align}

Now, taking $n=\ga^1$ and $k=\ga_e^2$ in \eqref{eq:1}, and $n=\ga^1$ and $k=\ga^2$ in \eqref{eq:2} and \eqref{eq:3}, yields
\[
\RGa=\begin{cases}\ga^1+\ga_e^2/2-1 & \h{if $(\Ga_1,\Ga_2) \in \Cb^1$,} \\ 2\ga^2-1 & \h{if $(\Ga_1,\Ga_2) \in \Cr^2$,} \\ 2\ga^1-1 & \h{if $(\Ga_1,\Ga_2) \in \Cb^3$.}\end{cases}
\]
Hence, by symmetry,
\[
\RGa=\begin{cases}\ga^2+\ga_e^1/2-1 & \h{if $(\Ga_1,\Ga_2) \in \Cr^1$,} \\ 2\ga^1-1 & \h{if $(\Ga_1,\Ga_2) \in \Cb^2$,} \\ 2\ga^2-1 & \h{if $(\Ga_1,\Ga_2) \in \Cr^3$.}\end{cases}
\]

Thus, if $(\Ga_1,\Ga_2) \in \Cr^1 \cup \Cr^2 \cup \Cr^3$, then $\RGa=\RrGa$, and if $(\Ga_1,\Ga_2) \in \Cb^1 \cup \Cb^2 \cup \Cb^3$, then $\RGa=\RbGa$; now apply \eqref{R-comp}.

\medskip

Consider now $(\Ga_1,\Ga_2) \in \Cr^4 \cup \Cb^4$; \WLOGtwo, assume that $(\Ga_1,\Ga_2) \in \Cb^4$. In order to simplify notation, let $n=\ga^1$ and $k=\ga_e^2$. By Lemma~\ref{Lemma 3.1},
\[
R(\C_{\geq n},\C_{\geq n} \cup \{C_k\}) \leq n+k/2-1.
\]
Hence, by Lemma~\ref{Lemma 2.1},
\[
R(\{C_n,C_{n+1}\},\{C_n,C_{n+1}\} \cup \{C_k\}) \leq n+k/2-1.
\]
Hence, by Lemma~\ref{Lemma 3.3},
\[
\RGa \leq R(\{C_n,C_{n+1}\},\{C_k\}) \leq n+k/2-1.
\]
Since, by \Pro~\ref{redblue}, $\RbGa=n+k/2-1$, and $\RbGa \leq \mm \leq \RGa$, this completes the proof.
\end{proof}

\begin{Prop} \label{C3C3orC4C4}
Let $(\Ga_1,\Ga_2) \in \Cc$ with $C_3$ or $C_4 \in \Ga_1 \cap \Ga_2$. Then
\[
\mGa=5
\]
and
\[
\RGa=\begin{cases}5 & \h{if $C_3$ and $C_5 \in \Ga_1 \cup \Ga_2$,} \\ 6 & \h{otherwise.}\end{cases}
\]
\end{Prop}

\begin{proof}
By \Co~\ref{maxredblue}, $\mGa=5$. Hence, by \eqref{R-comp}, $\RGa \geq 5$. By \eqref{ub} and since $R(C_3,C_3)=R(C_4,C_4)=6$, $\RGa \leq 6$.

If $C_3 \notin \Ga_1 \cup \Ga_2$, then Colouring~2 shows that $\RGa \geq 6$, and if $C_5 \notin \Ga_1 \cup \Ga_2$, then Colouring~3 shows that $\RGa \geq 6$. Thus, from now on, assume that both $C_3$ and $C_5$ belong to $\Ga_1 \cup \Ga_2$; \WLOGtwo, assume that $C_3 \in \Ga_1$. We have to show that $\RGa \leq 5$.

$C_3 \in \Ga_1 \cap \Ga_2$: It is well known and easy to verify that Colouring~3 is the only red-blue graph on $5$ vertices without monochromatic $3$-cycles, and it contains a red $C_5$ and a blue $C_5$.

$C_4 \in \Ga_1 \cap \Ga_2$: Let $G$ be an arbitrary red-blue graph on $5$ vertices, and assume $G$ is $(\Ga_1,\Ga_2)$-avoiding. Were each vertex of $G$ incident with exactly two edges of each colour, we would obtain a red $C_5$ and a blue $C_5$, whence some vertex of $G$ is incident with at least three edges of the same colour. Since $G$ contains no red $C_3$, it follows that $G$ contains a blue $C_3$, say $C=x_1x_2x_3x_1$; let $v_1$ and $v_2$ be the vertices of $G-V(C)$. Since $G$ contains no monochromatic $C_4$, we may assume that $v_1$ is red adjacent to $x_1$ and $x_3$, but blue adjacent to $x_2$, while $v_2$ is red adjacent to $x_2$ and $x_3$, but blue adjacent to $x_1$. Now, if $v_1v_2$ is red, then $v_1v_2x_3v_1$ is a red $C_3$, and if $v_1v_2$ is blue, then $v_1v_2x_1x_2v_1$ is a blue $C_4$, contrary to the hypothesis.
\end{proof}

\begin{Prop} \label{C4C3orC3C4}
Let $(\Ga_1,\Ga_2) \in \Cc$ with $C_4 \in \Ga_1 \not\ni C_3$ and $C_3 \in \Ga_2 \not\ni C_4$. Then
\[
\RGa=\mGa=\begin{cases}6 & \h{if $C_6 \in \Ga_2$,} \\ 7 & \h{otherwise.}\end{cases}
\]
\end{Prop}

\begin{proof}
The values of $\mGa$ follow from \Co~\ref{maxredblue}. Hence, by \eqref{R-comp},
\[
\RGa \geq \begin{cases}6 & \h{if $C_6 \in \Ga_2$,} \\ 7 & \h{otherwise.}\end{cases}
\]

We now turn to the upper bounds. By \eqref{ub}, $\RGa \leq R(C_4,C_3)=7$. Thus, from now on, assume that $C_6 \in \Ga_2$. We have to show that $\RGa \leq 6$. Thus, let $G$ be an arbitrary red-blue graph on $6$ vertices, and assume $G$ is $(\Ga_1,\Ga_2)$-avoiding. Since $R(C_3,C_3)=R(C_4,C_4)=6$, $G$ contains a red $C_3$ and a blue $C_4$, say $C=x_1x_2x_3x_4x_1$; let $v_1$ and $v_2$ be the vertices of $G-V(C)$. Then $x_1x_3$ and $x_2x_4$ are red, and, \WLOGtwo, either (1) $v_1v_2x_4v_1$ or (2) $v_2x_2x_4v_2$ is a red $C_3$.

\textbf{Case~1.} Since $G$ contains no red $C_4$, $x_2v_1$ is blue, since $G$ contains no blue $C_3$, $x_3v_1$ is red, and since $G$ contains no red $C_4$, $x_3v_2$ is blue. Now, if $x_2v_2$ is red, then $x_2v_2v_1x_4x_2$ is a red $C_4$, and if $x_2v_2$ is blue, then $x_2v_2x_3x_2$ is a blue $C_3$, contrary to the hypothesis.

\textbf{Case~2.} Since $G$ contains no red $C_4$, either $v_1x_2$ or $v_1x_4$ is blue. In either case, $v_1x_1$ and $v_1x_3$ are red (since $G$ contains no blue $C_3$). Thus $C_{(1)}=x_1v_1x_3x_1$ and $C_{(2)}=x_2v_2x_4x_2$ are red $3$-cycles. Since $G$ contains no red $C_4$, at most one of the edges between $C_{(1)}$ and $C_{(2)}$ is red, whence $G$ contains a blue $C_6$, contrary to the hypothesis.
\end{proof}

\begin{Prop} \label{g1>=5g2=3}
Let $(\Ga_1,\Ga_2) \in \Cc$ with $\ga^1 \geq 5$ and $\ga^2=3$. Then
\[
\RGa=\mGa.
\]
\end{Prop}

\begin{proof}
Put $n=\ga^1$. Since $\RbGa \leq \mGa \leq \RGa$, it suffices to prove that $\RGa \leq \RbGa$. Thus, let $G$ be an arbitrary red-blue graph on $\RbGa$ vertices, and assume $G$ is $(\Ga_1,\Ga_2)$-avoiding. Then $G$ contains an odd blue cycle. Let $C=x_1x_2 \cd x_{2k+1}x_1$ be a shortest odd blue cycle in $G$; note that $k \geq 2$. Were some chord of $C$ blue, $G$ would contain an odd blue cycle shorter than $C$, whence all chords of $C$ are red.

We now show that $G$ contains a red $C_n$, contrary to the hypothesis. If $2k+1 \leq n$, let $U$ be an $n$-subset of $V(G)$ such that $G[U]$ contains $C$. Then $G[U]$ is not blue bipartite, whence, by Lemma~\ref{Lemma X}, $G[U]$ contains a red $C_n$. On the other hand, if $2k+1>n$, then, by Lemma~\ref{chords}, $G[V(C)]$ contains a red $C_n$.
\end{proof}

\begin{Prop} \label{g1>=5g2=4}
Let $(\Ga_1,\Ga_2) \in \Cc$ with $\ga^1 \geq 5$ and $\ga^2=4$. Then
\[
\RGa=\mGa.
\]
\end{Prop}

\begin{proof}
If $\ga^1 \geq 6$, then $(\Ga_1,\Ga_2) \in \Cb^1$, and if $\ga^1=5$ and $\ga_e^1=6$, then $(\Ga_1,\Ga_2) \in \Cb^4$, as defined in the proof of \Pro~\ref{C1C2}. Hence, if $\ga^1 \geq 6$ or $\ga_e^1=6$, the result follows from \Pro~\ref{C1C2}. Otherwise, $\RGa \geq \mm=7$, by \Co~\ref{maxredblue}. Since $\RGa \leq R(C_5,C_4)=7$, this completes the proof.
\end{proof}

\begin{Prop} \label{g1=g2=5}
Let $(\Ga_1,\Ga_2) \in \Cc$ with $\ga^1=\ga^2=5$. Then
\[
\RGa=\mGa=\begin{cases}7 & \h{if $\max(\ga_e^1,\ga_e^2)=6$,} \\ 8 & \h{if $\max(\ga_e^1,\ga_e^2)=8$,} \\ 9 & \h{if $\max(\ga_e^1,\ga_e^2) \geq 10$.}\end{cases}
\]
\end{Prop}

\begin{proof}
The values of $\mGa$ follow from \Co~\ref{maxredblue}. Hence, by \eqref{R-comp},
\[
\RGa \geq \begin{cases}7 & \h{if $\max(\ga_e^1,\ga_e^2)=6$,} \\ 8 & \h{if $\max(\ga_e^1,\ga_e^2)=8$,} \\ 9 & \h{if $\max(\ga_e^1,\ga_e^2) \geq 10$.}\end{cases}
\]

We now turn to the upper bounds. By \eqref{ub}, $\RGa \leq R(C_5,C_5)=9$. We show that $\RGa \leq 7$ if $\max(\ga_e^1,\ga_e^2)=6$. To show that $\RGa \leq 8$ if $\max(\ga_e^1,\ga_e^2) \leq 8$ requires more work, but the techniques are similar.

Let $G$ be an arbitrary red-blue graph on $7$ vertices, and assume $G$ is $(\Ga_1,\Ga_2)$-avoiding. Since $R(C_4,C_4)=6$, $G$ contains a monochromatic $C_4$; say that $C=x_1x_2x_3x_4x_1$ is a blue $C_4$, and let $v_1$, $v_2$, and $v_3$ be the vertices of $G-V(C)$. Since $G$ contains no blue $C_5$, each $v_i$ is red adjacent to two opposite vertices of $C$. \WLOGone, assume that $v_1$ and $v_2$ are red adjacent to $x_1$ and $x_3$.

Since $G$ contains no red $C_5$, $v_3$ is blue adjacent to two opposite vertices of the red $C_4$ $v_1x_1v_2x_3v_1$. Thus, (1) $v_3x_1$ and $v_3x_3$ are red, in which case $v_3v_1$ and $v_3v_2$ are blue, (2) $v_3x_2$ and $v_3x_4$ are red, and $v_3x_1$ and $v_3x_3$ are blue, or (3) $v_3x_2$ and $v_3x_4$ are red, and $v_3v_1$ and $v_3v_2$ are blue. Note that Cases~1 and 2 are symmetric with respect to colour interchange. Thus we only have to consider Cases~1 and 3. In either case, were $v_1x_2$ and $v_2x_4$ blue, $v_1x_2x_1x_4v_2v_3v_1$ would be a blue $C_6$, whence at least one of them is red, say $v_1x_2$.

\textbf{Case~1.} Were $x_2v_2$ red, $x_2v_2x_1v_3x_3v_1x_2$ would be a red $C_6$, whence $x_2v_2$ is blue. Hence, were $x_4v_3$ blue, $x_4v_3v_2x_2x_1x_4$ would be a blue $C_5$, whence $x_4v_3$ is red. Now, if $x_4v_1$ is red, then $x_4v_1x_1v_2x_3v_3x_4$ is a red $C_6$, and if $x_4v_1$ is blue, then $x_4v_1v_3v_2x_2x_1x_4$ is a blue $C_6$, contrary to the hypothesis.

\textbf{Case~3.} Were $x_4v_2$ red, $x_4v_2x_1v_1x_2v_3x_4$ would be a red $C_6$, whence $x_4v_2$ is blue. Now, if $x_1v_3$ is red, then $x_1v_3x_2v_1x_3v_2x_1$ is a red $C_6$, and if $x_1v_3$ is blue, then $x_1v_3v_2x_4x_3x_2x_1$ is a blue $C_6$, contrary to the hypothesis.
\end{proof}

\begin{Prop} \label{g1>=6g2=5}
Let $(\Ga_1,\Ga_2) \in \Cc$ with $\ga^1 \geq 6$ and $\ga^2=5$. Then
\[
\RGa=\mGa.
\]
\end{Prop}

\begin{proof}
Put $n=\ga^1$. Since $\RbGa \leq \mGa \leq \RGa$, it suffices to prove that $\RGa \leq \RbGa$. Thus, let $G$ be an arbitrary red-blue graph on $\RbGa$ vertices, and assume $G$ is $(\Ga_1,\Ga_2)$-avoiding. Then $G$ contains an odd blue cycle. Furthermore, by \Pro~\ref{redblue}, $|G| \geq n+2$. Hence, it follows from \Pro~\ref{new} that $G$ contains a red $C_n$, contrary to the hypothesis.
\end{proof}

\begin{Prop} \label{g1>=6g2=6}
Let $(\Ga_1,\Ga_2) \in \Cc$ with $\ga^1 \geq 6$ and $\ga^2=6$. Then
\[
\RGa=\mGa.
\]
\end{Prop}

\begin{proof}
If $\ga^1 \neq 7$, then $(\Ga_1,\Ga_2) \in \Cb^1$, if $\ga^1=7$ and $\ga_e^1 \geq 12$, then $(\Ga_1,\Ga_2) \in \Cr^2$, and if $\ga^1=7$ and $\ga_e^1=8$, then $(\Ga_1,\Ga_2) \in \Cb^4$, as defined in the proof of \Pro~\ref{C1C2}. Hence, if $(\ga^1,\ga_e^1) \neq (7,10)$, the result follows from \Pro~\ref{C1C2}. Otherwise, $\RGa \geq \mm=10$, by \Co~\ref{maxredblue}, whence we have to show that $\RGa \leq 10$. Thus, let $G$ be an arbitrary red-blue graph on $10$ vertices, and assume $G$ is $(\Ga_1,\Ga_2)$-avoiding.

Since $R(C_8,C_6)=10$, $G$ contains a red $C_8$, say $C=x_1y_3x_2y_4x_3y_1x_4y_2x_1$; let $X=\{x_i\}$ and $Y=\{y_i\}$, and let $v_1$ and $v_2$ be the vertices of $G-V(C)$. Since $G$ contains no red $C_7$, $x_1x_2x_3x_4x_1$ and $y_1y_2y_3y_4y_1$ are blue $4$-cycles. Either all edges between $X$ and $Y$ are red, or at least one of them, say $x_1y_1$, is blue. In the latter case, $x_2y_2$, $x_3y_3$, and $x_4y_4$ are red (since $G$ contains no blue $C_6$). Hence, in either case, $x_1x_3$, $x_2x_4$, $y_1y_3$, and $y_2y_4$ are blue (since $G$ contains no red $C_7$). Therefore, two disjoint blue edges between $X$ and $Y$ would yield a blue $C_6$, whence all blue edges between $X$ and $Y$ have a common vertex, which we may assume is $x_1$.

Were some $v_j$ red adjacent to both $X$ and $Y$ (i.e., to at least one vertex in $X$ and at least one vertex in $Y$), $G$ would contain a red $C_7$, and if both $v_1$ and $v_2$ were blue adjacent to either all $x_i$ or all $y_i$, $G$ would contain a blue $C_6$. Hence, we may assume that $v_1$ is blue adjacent to all $x_i$ while $v_2$ is blue adjacent to all $y_i$. Now, if $v_1v_2$ is red, then, since $v_2$ is red adjacent to $X$ and $v_1$ is red adjacent to $Y$, we obtain a red $C_{10}$. On the other hand, if $v_1v_2$ is blue, then if $v_2$ is blue adjacent to $X$ or $v_1$ is blue adjacent to $Y$, we obtain a blue $C_6$, and if $v_2$ is red adjacent to all $x_i$ and $v_1$ is red adjacent to all $y_i$, we obtain a red $C_{10}$, contrary to the hypothesis.
\end{proof}

We are now ready to prove \Th~\ref{MT}.

\begin{proof}[Proof of \Th~\ref{MT}]
The first claim follows from \eqref{R-comp} and \Pro~\ref{C3C3orC4C4}, and the second claim is an immediate consequence of \Pros~\ref{C1C2}--\ref{g1>=6g2=6}.
\end{proof}

\subsection{Structural results} \label{struc}

The main goal of this subsection is to state and prove \Th~\ref{struc35+}, the unabridged version of \Th~\ref{struc35}. We first prove that many $(\Ga_1,\Ga_2)$-avoiding graphs, not only critical ones, have to be blue bipartite.

\begin{Prop} \label{avoiding}
Given $(\Ga_1,\Ga_2) \in \Cc$ and a $(\Ga_1,\Ga_2)$-avoiding graph $G$, assume that either
\[
\h{$\ga^1 \geq 5$, $\ga^2=3$, and $|G| \geq \ga^1$}
\]
or
\[
\h{$\ga^1 \geq 6$, $\ga^2=5$, and $|G| \geq \ga^1+2$.}
\]
Then $G$ is blue bipartite.
\end{Prop}

\begin{proof}
In order to obtain a contradiction, suppose that $G$ is not blue bipartite, and let $n=\ga^1$.

Assume first that $n \geq 5$, $\ga^2=3$, and $|G| \geq n$, and let $C$ be a shortest odd blue cycle in $G$. Then it follows, in the same way as in the proof of \Pro~\ref{g1>=5g2=3}, that $G$ contains a red $C_n$, contrary to the $(\Ga_1,\Ga_2)$-avoidance of $G$. Thus $G$ is blue bipartite.

Assume now that $n \geq 6$, $\ga^2=5$, and $|G| \geq n+2$. Then it follows from \Pro~\ref{new} that $G$ contains a red $C_n$. Thus $G$ is blue bipartite.
\end{proof}

We find this surprising, since $\RGa$ can be much larger than $\ga^1$ and $\ga^1+2$. For instance (with $\ga^1$ and $\ga^2$ as in \Pro~\ref{avoiding}), $R(C_{\ga^1},C_{\ga^2})=2\ga^1-1$. Hence, even if $|G|$ is only about half the Ramsey number, $G$ has to be blue bipartite in order to avoid both a red $C_{\ga^1}$ and a blue $C_{\ga^2}$.

Note that if $(\ga_2,\ga_e^2)=(5,6)$, then $\RGa=\ga^1+2$, whence \Pro~\ref{avoiding} does not tell us anything. This is why in \Th~\ref{struc35+}, we require that $\ga_e^2 \geq 8$.

The next result characterises all $(\Ga_1,\Ga_2)$-avoiding, blue bipartite graphs on the maximum number of vertices, for any sets of cycles.

\begin{Prop} \label{critical}
Given $(\Ga_1,\Ga_2) \in \Cc$, let $G$ be a blue bipartite graph on $\RbGa-1$ vertices; say that $\Gb \seq K_{p,q}$, where $p+q=|G|$ and $p \geq q$, and let $n=\ga^1$ and $k=\ga_e^2$.

If $2n>k$ and $(n,k) \neq (3,4)$, then $G$ is $(\Ga_1,\Ga_2)$-avoiding if and only if either
\begin{itemize}
  \item[(1)] $p=n-1$, $q=k/2-1$, and all red edges (if any) between $K_{n-1}$ and $K_{k/2-1}$ are incident with a common vertex in $K_{n-1}$,
\end{itemize}
or
\begin{itemize}
  \item[(2)] $p=n-2$, $q=k/2$, there is a vertex $x \in K_{k/2}$ such that there are $n-3$ or $n-2$ red edges between $x$ and $K_{n-2}$, and all other edges between $K_{n-2}$ and $K_{k/2}$ are blue.
\end{itemize}

If $2n \leq k$ or $(n,k)=(3,4)$, then $G$ is $(\Ga_1,\Ga_2)$-avoiding if and only if $p=q=n-1$ and, moreover, if $n \geq 4$, there is at most one red edge between the two $K_{n-1}$, while if $n=3$, there are zero (only possible if $k \geq 6$), one, or two disjoint (only possible if $\ga_e^1 \geq 6$) red edges between the two $K_2$.
\end{Prop}

\begin{proof}
It is easily seen that the above conditions are sufficient for $G$ to be $(\Ga_1,\Ga_2)$-avoiding. Hence, we have to prove that they are necessary.

$2n>k$ and $(n,k) \neq (3,4)$: We have $|G|=\RbGa-1=n+k/2-2$. Were $p \geq n$, $G$ would contain a red $C_n$, whence either (1) $p=n-1$ and $q=k/2-1$, or (2) $p \leq n-2$ and $q \geq k/2$.

\textbf{Case~1.} Were there two disjoint red edges between $K_{n-1}$ and $K_{k/2-1}$, $G$ would contain a red $C_n$ (since $n \geq 4$), whence all red edges between $K_{n-1}$ and $K_{k/2-1}$ have a common vertex $x$. Either $x \in K_{k/2-1}$ and there is at most one red edge between $x$ and $K_{n-1}$, or $x \in K_{n-1}$ and there are up to $k/2-1$ red edges between $x$ and $K_{k/2-1}$.

\textbf{Case~2.} Were $G$ blue complete bipartite, $G$ would contain a blue $C_k$, whence there is at least one red edge between $K_p$ and $K_q$. As in Case~1, all red edges between $K_p$ and $K_q$ have a common vertex $x$. Since $G$ contains no blue $C_k$, $p=n-2$, $q=k/2$, $x \in K_{k/2}$, and there are $n-3$ or $n-2$ red edges between $x$ and $K_{n-2}$.

$2n \leq k$ or $(n,k)=(3,4)$: In this case, $|G|=2n-2$. Were $p \geq n$, $G$ would contain a red $C_n$, whence $p=q=n-1$. Since $G$ contains no red $C_n$, no vertex is incident with more than one red edge between the two $K_{n-1}$. Therefore, if $n \geq 4$, then there is at most one red edge between the two $K_{n-1}$. On the other hand, if $n=3$, then zero red edges is impossible if $k=4$, and two disjoint red edges is impossible if $\ga_e^1=4$.
\end{proof}

Observe that if $\RbGa=\RGa$ (as is the case in \Th~\ref{struc35+}), then \Pro~\ref{critical} characterises all $(\Ga_1,\Ga_2)$-critical graphs that are blue bipartite. On the other hand, if $\RbGa<\RGa$, then no $(\Ga_1,\Ga_2)$-critical graph is blue bipartite. In this case, if \Con~\ref{MC} is true and unless $(\Ga_1,\Ga_2) \in \mathscr{C}$ (as defined above \Th~\ref{MT}), then there is a red bipartite, $(\Ga_1,\Ga_2)$-critical graph, and the obvious ``red bipartite version'' of \Pro~\ref{critical} characterises all of them.

We can now state and prove the unabridged version of \Th~\ref{struc35}.

\begin{Thm} \label{struc35+}
Given $(\Ga_1,\Ga_2) \in \Cc$, let $n=\ga^1$ and $k=\ga_e^2$, and assume that either
\begin{equation} \label{eq1+}
\h{$n \geq 5$ and $\ga^2=3$}
\end{equation}
or
\begin{equation} \label{eq2+}
\h{$n \geq 6$, $\ga^2=5$, and $k \geq 8$.}
\end{equation}
Then $G$ is $(\Ga_1,\Ga_2)$-critical if and only if $\Gb \seq K_{p,q}$, where $p+q=|G|$ and $p \geq q$, and, moreover, if $2n>k$, either \emph{(1)} or \emph{(2)} in \Pro~\ref{critical} holds, while if $2n \leq k$, $p=q=n-1$ and there is at most one red edge between the two $K_{n-1}$.
\end{Thm}

\begin{proof}
Since either \eqref{eq1+} or \eqref{eq2+} holds, $\RGa=\RbGa$. As noted above, \Pro~\ref{critical} therefore characterises all $(\Ga_1,\Ga_2)$-critical graphs that are blue bipartite. However, if \eqref{eq1+} holds, then $\RbGa \geq \ga^1+1$, so if $G$ is $(\Ga_1,\Ga_2)$-critical, then $|G|=\RGa-1 \geq \ga^1$. Similarly, if \eqref{eq2+} holds, then $\RbGa \geq \ga^1+3$, so if $G$ is $(\Ga_1,\Ga_2)$-critical, then $|G| \geq \ga^1+2$. Hence, by \Pro~\ref{avoiding}, all $(\Ga_1,\Ga_2)$-critical graphs are blue bipartite. This completes the proof.
\end{proof}

As noted in the introduction, \Co~\ref{cor3} is a direct consequence of \Th~\ref{struc35+}.

\medskip

We end this paper with some comments regarding star-critical Ramsey numbers.

Given a graph $G$ and a subgraph $H$ of $G$, let $G-H$ be the subgraph of $G$ obtained by deleting the edges of $H$. The \emph{star-critical Ramsey number} $r_*(G_1,G_2)$ was introduced by Hook and Isaak~\cite{Hook,H-I}. It is the smallest integer $k$ such that each red-blue colouring of the edges of $K_n-K_{1,n-1-k}$, where $n=R(G_1,G_2)$, contains a red copy of $G_1$ or a blue copy of $G_2$. That is, it describes the largest star whose edges can be removed from $K_n$ so that the resulting subgraph is still forced to contain a red $G_1$ or a blue $G_2$.

Let $\De$ be the set of all pairs $(n,k)$ of integers such that $n \geq k \geq 3$, $k$ is odd, and $(n,k) \neq (3,3)$. Zhang, Broersma, and Chen~\cite{Z-B-C} recently proved the following:

\begin{Thm}[\mbox{\cite[\Th~4]{Z-B-C}}] \label{Z-B-C}
Let $(n,k) \in \De$. Then $r_*(C_n,C_k)=n+1$.
\end{Thm}

The lower bound in \Th~\ref{Z-B-C} follows from an easy construction. The proof of the upper bound, on the other hand, is more complicated. We observe that the upper bound follows easily from \Co~\ref{cor3} in the special cases $k=3$ with $n \geq 5$, and $k=5$ with $n \geq 6$. We demonstrate this when $k=3$; the proof when $k=5$ is completely analogous.

Let $r=R(C_n,C_3)=2n-1$. We have to show that every colouring of $G=K_r-K_{1,r-1-(n+1)}=K_{2n-1}-K_{1,n-3}$ contains a red $C_n$ or a blue $C_3$. Let $v$ be the centre vertex of the star $K_{1,n-3}$, and put $H=K_{2n-1}-v$ (so $H$ is a $K_{2n-2}$). Since $(2n-2)-(n-3)=n+1$, $v$ is adjacent to $n+1$ vertices of $H$. Colour $G$ arbitrarily, and assume $G$ becomes $(C_n,C_3)$-avoiding. By \Co~\ref{cor3}, $H_{\mathrm{blue}}=K_{n-1,n-1}$ or $K_{n-1,n-1}-e$ for some edge $e$. Since $G$ contains no red $C_n$, $v$ has at most one red edge to each red $K_{n-1}$. Hence, $v$ has at least $(n+1)-2=n-1$ blue edges. Now, if all blue edges go to the same red $K_{n-1}$, we obtain a red $C_n$, and if not, we obtain a blue $C_3$, contrary to the hypothesis. Thus, $r_*(C_n,C_3) \leq n+1$.

In fact, for all pairs $(n,k) \in \De$ such that $G$ is $(C_n,C_k)$-critical if and only if $\Gb=K_{n-1,n-1}$ or $K_{n-1,n-1}-e$ for some edge $e$ (see \Que~\ref{ques}), a completely analogous argument proves the upper bound in \Th~\ref{Z-B-C}.

\section*{Acknowledgements}

The author wants to thank J\"orgen Backelin for valuable comments and discussions, and Axel Hultman and Stanis{\l}aw P.~Radziszowski for comments which much improved the presentation.

\bibliographystyle{amsplain}

\bibliography{Referenser_bara_initialer}

\end{document}